\documentclass[a4paper]{amsart}

\newcommand{\rn}{\mathbb{R}^{N}}
\newcommand{\RR}{\mathbb{R}}
\newcommand{\N}{\mathbb{N}}

\newcommand{\R}{\mathcal{R}}

\renewcommand{\lq}{L^q(\Omega)}
\renewcommand{\wp}{W^{1,p}(\Omega)}
\newcommand{\wpc}{W^{1,p}_0(\Omega)}

\newcommand{\rd}{\mathrm{d}}
\newcommand{\ep}{\varepsilon}

\def\div{\mathop{\mbox{\normalfont div}}\nolimits}

\newtheorem{te}{Theorem}[section]
\newtheorem{lem}[te]{Lemma}
\newtheorem{co}[te]{Corollary}
\newtheorem{pr}[te]{Proposition}

\theoremstyle{remark}
\newtheorem{ob}[te]{Remark}

\theoremstyle{definition}
\newtheorem{de}[te]{Definition}

\numberwithin{equation}{section}

\title[Eigenvalue optimization]{An optimization problem for the first weighted eigenvalue problem plus a potential}

\author[L. Del Pezzo and J. Fernandez Bonder]
{Leandro M. Del Pezzo and Juli\'an Fern\'andez Bonder}

\address{Departamento  de Matem\'atica, FCEyN, Universidad de Buenos Aires, \hfill\break\indent Pabell\'on I, Ciudad Universitaria (1428), Buenos Aires, Argentina.}

\email[L. Del Pezzo]{ldpezzo@dm.uba.ar}

\email[J. Fernandez Bonder]{jfbonder@dm.uba.ar}
\urladdr[J. Fernandez Bonder]{http://mate.dm.uba.ar/~jfbonder}

\thanks{Supported by Universidad de Buenos Aires under grant X078 and by ANPCyT PICT2006--290. J. Fern\'andez Bonder is a member of CONICET. Leandro Del Pezzo is a fellow of CONICET}



\begin{document}

\begin{abstract}
In this paper, we study the problem of minimizing the first eigenvalue of the $p-$Laplacian plus a potential with weights, when the potential and the weight are allowed to vary in the class of rearrangements of a given fixed potential $V_0$ and weight $g_0$. Our results generalized those obtained in \cite{DPFB1} and \cite{CEP2}.
\end{abstract}

\maketitle


\section{Introduction}

In this paper we consider the following nonlinear eigenvalue problem with weights
\begin{equation}\label{eigen}
\begin{cases}
-\Delta_p u + V(x) |u|^{p-2}u = \lambda g(x) |u|^{p-2}u & \mbox{in }\Omega,\\
u=0 & \mbox{on }\partial\Omega,
\end{cases}
\end{equation}
where $\Omega$ is a smooth bounded open subset of $\rn$. Here $\Delta_p u := \div(|\nabla u|^{p-2}\nabla u)$ is the well-known $p-$Laplace operator, $V$ is a potential function and $g$ is a weight.

Our aim is to study the following optimization problems:
\begin{equation}\label{problema}
I:=\inf\left\{\lambda(g,V)\colon g\in\R(g_0), V\in\R(V_0)\right\}.
\end{equation}
where $V_0$ and $g_0$ are fixed potential and weight functions respectively with some precise hypotheses that we state below (see (H1) and (H2)) and $\R(V_0)$, $\R(g_0)$ are the classes of rearrangements of $V_0$ and $g_0$ respectively.

This type of optimization problems for eigenvalues of the $p-$Laplacian have deserved a great deal of attention. We like to mentioned the work of \cite{AsHa} where the problem was analyzed in the context of the classical Laplacian ($p=2$) without weights ($g\equiv 1$) and the potential was allowed to vary in the unit ball of some $L^q(\Omega)$.

Later on, the results in \cite{AsHa} were extended to the nonlinear case in \cite{DPFB1}, again without weights.

A related minimization problem when the minimization parameter was allowed to vary in the class of rearrangements of a fixed function, was first considered by \cite{CEP}. See also \cite{DPFB2}.

The eigenvalue problem \eqref{eigen} was analyzed exhaustively in \cite{CRQ} where the authors prove the existence of a principal eigenvalue and several properties of it. The results of \cite{CRQ}  closely related to our work are discussed in Section 2.

More recently, in \cite{CEP2}, the authors analyze problem \eqref{problema} but when the potential function is zero. In that work the authors prove the existence of a minimizing weight $g_*$ in the class of rearrangements of a fixed function $g_0$ and, in the spirit of \cite{Bu1} they found a sort of {\em Euler-Lagrange formula} for $g_*$. However, this formula does not appear to be suitable for use in actual computations of these minimizers.

In this work we first extend the results in \cite{CEP2} to \eqref{eigen} and prove the existence of a minimizing weight and potential for \eqref{problema}. Also the same type of Euler-Lagrange formula is proved for both the weight and potential. But, we go further and study the dependence of the eigenvalue $\lambda(g,V)$ with respect to $g$ and $V$ and prove the continuous dependence in $L^q$ norm and, moreover, the differentiability with respect to regular perturbations of the weight and the potential.

In the case when the perturbations are made inside the class of rearrangements, we exhibit a simple formula for the derivative of the eigenvalue with respect to $g$ and $V$.

We believe that this formula can be used in actual computations of the optimal eigenvalue, weight and potential, since this type of formulas have been used in similar problems in the past with significant success, see \cite{FBGR, HENROT, Oudet, Pironneau} and references therein. This is what we think is the main contribution of our paper.

\subsection*{Organization of the paper}
After finishing this introduction, the paper is organized as follows. In Section 2 we collect some preliminaries needed in the paper. First we discuss the results of \cite{C,CRQ} on the eigenvalue problem \eqref{eigen} and second we recall some known results on rearrangements due to \cite{Bu1, Bu2}. In Section 3 we prove the existence of a unique minimizer and give a characterization of it, similar to the one found in \cite{CEP2} for the problem without potential. Finally, in Section 4 we study the dependence of the eigenvalue with respect to the weight and the potential and prove, first the continuous dependence in the $L^q$ topology (Proposition \ref{continuidad}) and finally we show a simple formula for the derivative of the eigenvalue with respect to regular variations of the weight and the potential within the class of rearrangements (Theorem \ref{teo.principal}).


\section{Preliminaries}

\subsection{Properties of the principal eigenvalue}
Let $\Omega$ be a bounded smooth domain in $\rn$ with $N\ge2$ and $1<p<\infty.$ Let $g_0$ and $V_0$ be measurable functions that satisfy the following assumptions:
\begin{eqnarray*}	
\textrm{(H1)}&\qquad& g_0, V_0 \in\lq \textrm{ where }
\begin{cases}
q>\frac{N}{p} &\textrm{if } 1<p\le N,\\
q=1 &\textrm{if } p>N,
\end{cases}\\
\textrm{(H2)}&\qquad& \|V_0^{-}\|_{\lq}< S_{pq'} \textrm{ or } V \ge -S_p + \delta \textrm{ for some } \delta>0 \textrm{ and } g_0^+\not\equiv 0,
\end{eqnarray*}
where $f^{-}=\min\{f,0\}$, $f^+ = \max\{f,0\}$ and  $S_r$ ($r=p,\, pq'$) is the best (largest) constant in the Sobolev--Poincar\'e inequality
$$
S \|u\|_{L^{r}(\Omega)}^p \le \int_\Omega|\nabla u|^p \, \rd x \quad \forall u\in\wpc,
$$
i.e.,
$$
S_r := \inf\left\{\int_\Omega |\nabla u|^p \, \rd
x\colon u\in\wpc, \, \|u\|_{L^{r}(\Omega)}=1\right\}.
$$

Observe that if $g$ and $V$ are measurable functions that
satisfy there exists a unique positive principal eigenvalue $\lambda(g,V)$ of \eqref{eigen} and it is characterized by
\begin{equation}\label{caract}
\lambda(g,V) := \min\left\{\int_\Omega |\nabla u|^p + V(x)|u|^p\, \rd x\colon u\in\wpc, \int_\Omega g(x)|u|^p \, \rd x = 1\right\}.
\end{equation}
See \cite{CRQ}. Obviously, if $u$ is a minimizer, so is $|u|$; therefore we may assume $u\ge 0$.

The following Lemma is taken from \cite{C} and gives us the positivity of eigenfunctions associated to the principal eigenvalue.
\begin{lem}[\cite{C}, Proposition 3.2]\label{positivo}
Let $g$ and $V$ be two measurable functions that satisfy the assumption (H1). If $u\in\wpc$ is a nonnegative weak solution to \eqref{eigen} then either $u\equiv 0$ or $u>0$ for all $x\in \Omega$.
\end{lem}

\begin{proof}
The proof is a direct consequence of Harnack's inequality. See \cite{Serrin}.
\end{proof}

We therefore immediately obtain,
\begin{co}
Under the assumptions of the previous Lemma, every eigenfunction associated to the principal positive eigenvalue has constant sign.
\end{co}

Furthermore, following \cite{CRQ}, we have that the principal eigenvalue $\lambda(g,V)$ is simple.

\begin{lem}\label{simple}
Let $g$ and $V$ be two measurable functions that satisfy the assumption (H1). Let $u$ and $v$ be two eigenfunctions associated to $\lambda(g,V)$. Then, there exists a constant $c\in\RR$ such that $u=cv$.
\end{lem}

\begin{proof}
The proof follows immediately from Lemma 4 in \cite{CRQ}.
\end{proof}

\subsection{Results on Rearrangements}

We will now give some well-known results concerning the
rearrangements of functions. They can be found, for instance, in \cite{Bu1, Bu2}.

\begin{de} 
Given two functions $f,g:\Omega\to \RR$ measurable we say that $f$ is a rearrangement of $g$ if
$$
|\{x\in\Omega\colon f(x)\ge\alpha\}| = |\{x\in\Omega\colon g(x)\ge\alpha\}| \quad \forall \alpha \in \RR,
$$
where $|\cdot|$ denotes the Lebesgue measure.
\end{de}

Now, given $f_0\in L^p(\Omega)$ the set of all rearrangements of $f_0$ is denoted by $\R(f_0)$ and $\overline{\R(f_0)}$ denotes the clousure of $\R(f_0)$ in $L^p(\Omega)$ with respect to the weak topology.

\begin{te}\label{desrearr}
Let $1 \le p < \infty$ and let $p'$ be the conjugate exponent of $p$. Let $f_0 \in L^p(\Omega)$, $f_0\not\equiv 0$ and let $g\in L^{p'}(\Omega)$. Then, there exists $f_*, f^*\in\R(f_0)$ such that
$$
\int_\Omega f_*g \, \rd x \le \int_\Omega fg \, \rd x \le
\int_\Omega f^*g \, \rd x \quad \forall f \in \overline{\R(f)}.
$$
\end{te}

\begin{proof}
The proof follows from Theorem 4 in \cite{Bu1}.
\end{proof}

\begin{te}\label{unicrearr}
Let $1\le p \le\infty$ and let $p'$ be the conjugate of $p$. Let $f_0 \in L^p(\Omega),$ $f_0\not\equiv 0$ and let $g\in L^{p'}(\Omega)$.

If the linear functional $L(f)=\int_\Omega fg \, \rd x$ has a unique maximizer $f^*$ relative to $\R(f_0)$ then there exists an increasing function $\phi$ such that $f^*=\phi\circ g$ a.e. in $\Omega$.

Furthermore, if the linear functional $L(f)$ has a unique minimizer $f_*$ relative to $\R(f_0)$ then there exists a
decreasing function $\psi$ such that $f_*=\psi\circ g$ a.e. in $\Omega$.
\end{te}

\begin{proof}
The proof follows from Theorem 5 in \cite{Bu1}.
\end{proof}


\section{Minimization and Characterization}

Let $\Omega$ be a bounded smooth domain in $\rn$ with $N\ge2$ and $1<p<\infty.$ Given $g_0$ and $V_0$ measurable functions that satisfy the assumptions (H1) and (H2) our aim in this section is to analyze the following problem
$$
I = \inf\left\{\lambda(g,V)\colon g \in \R(g_0),\, V \in \R(V_0) \right\},
$$
where $\R(g_0)$ (resp. $\R(V_0)$) is the set of all
rearrangements of $g_0$ (resp. $V_0$) and $\lambda(g,V)$ is the first positive principal eigenvalue of problem \eqref{eigen}.

\begin{ob}
Observe that if $g\in\R(g_0)$ and \mbox{$V\in\R(V_0)$} then $g$ and $V$ satisfy (H1) and (H2).
\end{ob}

We first need a lemma to show that, under hypotheses (H1) and (H2), the functionals
$$
J_V(u) := \int_\Omega |\nabla u|^p\, \rd x + \int_\Omega V(x) |u|^p\, \rd x
$$
are uniformly coercive for $V\in\R(V_0)$.
\begin{lem}\label{coercividad}
Let $V_0$ satisfies (H1) and (H2). Then, there exists $\delta_0>0$ such that
$$
J_V(u) \ge \delta_0 \int_\Omega |\nabla u|^p\, \rd x,\quad \mbox{for every } V\in\R(V_0).
$$
\end{lem}

\begin{proof}
We prove the lemma assuming that $\|V_0^-\|_{L^q(\Omega)} < S_{pq'}$. Also, we assume that $1<p\le N$. The other cases are easier and are left to the reader.

First, observe that
$$
J_V(u) \ge \int_\Omega |\nabla u|^p\, \rd x + \int_\Omega V^-(x) |u|^p\, \rd x.
$$
On the other hand, $q>N/p$ implies that $pq'<p^*$. So
$$
\int_\Omega |V^-(x)| |u|^p\, \rd x\le \|V^-\|_{\lq} \|u\|_{L^{pq'}(\Omega)}^p = \|V_0^-\|_{\lq} \|u\|_{L^{pq'}(\Omega)}^p.
$$
Then, by (H2), there exists $\delta_0$ such that
$$
\|V_0^-\|_{\lq} \le (1-\delta_0) S_{pq'}.
$$
Therefore
$$
J_V(u)\ge \delta_0 \int_\Omega |\nabla u|^p\, \rd x,
$$
as we wanted to prove.
\end{proof}

\begin{ob}
We remark that what is actually needed is the uniform coercitivity of the functionals $J_V$ for $V\in \R(V_0)$. Hypotheses (H1) and (H2) are a simple set of hypotheses that guaranty that.
\end{ob}

We now prove that the infimum is achieved.
\begin{te}\label{existencia} 
Let $g_0$ and $V_0$ be measurable functions that satisfy the assumptions (H1) and (H2) and let $\R(g_0)$ and $\R(V_0)$ be the sets of all rearrangements of $g_0$ and 
$V_0$ respectively. Then there exists $g^*\in\R(g_0)$ and $V_*\in\R(V_0)$  such that
$$
I = \lambda(g^*,V_*).
$$
\end{te}

\begin{proof}
Let $\{(g_n, V_n)\}_{n\in\N}$ be a minimizing sequence, i.e.,
$$
g_n\in\R(g_0) \textrm{ and } V_n\in\R(V_0) \quad \forall n\in\N
$$
and
$$
I = \lim_{n\to\infty}\lambda(g_n,V_n).
$$
Let $u_n$ be the positive eigenfunction corresponding to 
$\lambda(g_n,V_n)$ then
\begin{equation}\label{uno}
\int_\Omega g_n(x) u_n^p = 1 \quad \forall n\in\N,
\end{equation}
and
$$
\lambda(g_n,V_n)=\int_\Omega |\nabla u_n|^p + V_n(x)u_n^p \, \rd x \quad \forall n \in \N.
$$
Hence
\begin{equation}\label{infimo}
I = \lim_{n\to\infty}\int_\Omega |\nabla u_n|^p + V_n(x) u_n^p \, \rd x.
\end{equation}
Thus, by Lemma \ref{coercividad}, $\{u_n\}_{n\in\N}$ is bounded in $\wpc$ and therefore there exists $u\in\wpc$ and some subsequence of $\{u_n\}_{n\in\N}$ (still denoted by $\{u_n\}_{n\in\N}$) such that
\begin{eqnarray}\label{weakly}
u_n&\rightharpoonup&u \quad \textrm{weakly in } \wp,\\
\label{strongly}
u_n&\to&u \quad \textrm{strongly in } L^{pq'}(\Omega).
\end{eqnarray}
Recall that our assumptions on $q$ imply that $pq'<p^*$.

On the other hand, $g_n\in\R(g_0)$ and $V_n\in\R(V_0)$ for all $n\in\N$ then
$$
\|g_n\|_{\lq} = \|g_0\|_{\lq} \textrm{ and } \|V_n\|_{\lq} = \|V_0\|_{\lq} \quad \forall n\in\N.
$$
Therefore there exists $f$, $W \in \lq$ and subsequences of $\{g_n\}_{n\in\N}$ and $\{V_n\}_{n\in\N}$ (still call by $\{g_n\}_{n\in\N}$ and $\{V_n\}_{n\in\N}$) such that
\begin{eqnarray}
\label{gn}
g_n &\rightharpoonup&f \quad \textrm{weakly in }\lq,\\
\label{vn}
V_n &\rightharpoonup&W \quad \textrm{weakly in }\lq.
\end{eqnarray}
Thus, by \eqref{infimo}, \eqref{weakly}, \eqref{strongly} and \eqref{vn}, we have that
$$
I \ge \int_\Omega|\nabla u|^p + W(x)|u|^p \, \rd x
$$
and by \eqref{uno}, \eqref{strongly} and \eqref{gn} we get
$$
\int_\Omega f(x)|u|^p \, \rd x = 1.
$$

Now, since $f\in\overline{\R(g_0)}$ and $W \in \overline{\R(V_0)}$, by Theorem \ref{desrearr}, there
exists $g^* \in \R(g_0)$ and $V_*\in\R(V_0)$ such that
$$
\alpha = \int_\Omega g^*(x)|u|^p\, \rd x \ge \int_\Omega f(x)u^p\, \rd x = 1
$$
and
$$
\int_\Omega V_*(x)u^p\, \rd x \le\int_\Omega W(x)|u|^p\, \rd x.
$$
Let $v = \alpha^{-1/p}|u|$, then
$$
\int_\Omega g^*(x)v^p \rd x = 1
$$
and
$$
\int_\Omega |\nabla v|^p + V_*(x)v^p \, \rd x
= \frac{1}{\alpha} \int_\Omega |\nabla u|^p + V_*(x)|u|^p
\, \rd x \le \frac{1}{\alpha}\int_\Omega |\nabla u|^p + W(x)|u|^p\, \rd x.
$$
Consequently
$$
\lambda(g^*,V_*) \le I,
$$
then
$$
I = \lambda(g^*,V_*).
$$
The proof is now complete.
\end{proof}

Now we give a characterization of $g^*$ and $V_*.$
\begin{te} Let $g_0$ and $V_0$ be measurable functions
that satisfy the assumptions (H1) and (H2). Let $g^*\in\R(g_0)$ and $V_*\in\R(V_0)$ be such that $\lambda(g^*,V_*) = I$ are the ones given by Theorem \ref{existencia}. Then there exist an increasing function $\phi$ and a decreasing function $\psi$ such that
\begin{align*}
g^*=\phi(u_*) &\quad \mbox{a.e. in } \Omega,\\
V_*=\psi(u_*) &\quad \mbox{a.e. in } \Omega,
\end{align*}
where $u_*$ is the positive eigenfunction associated to
$\lambda(g^*,V_*)$.
\end{te}

\begin{proof} We proceed in four steps

{\em Step 1}. First we show that $V_*$ is a minimizer of the linear functional
$$
L(V) := \int_\Omega V(x)u_*^p \, \rd x
$$
relative to $V\in\overline{\R(V_0)}$.

We have that
$$
\int_\Omega g^*(x)u_*^p \, \rd x = 1
$$
and
$$
I = \lambda(g^*,V_*)=\int_\Omega |\nabla u_*|^p + V_*(x)u_*^p \, \rd x,
$$
then, for all $V\in\R(V_0)$,
$$
\int_\Omega |\nabla u_*|^p + V_*(x)u_*^p \, \rd x \le \lambda(g^*, V)\le \int_\Omega |\nabla u_*|^p + V(x)u_*^p\, \rd x
$$
and therefore
$$
\int_\Omega V_*(x)u_*^p \, \rd x \le \int_\Omega V(x)u_*^p
\, \rd x \quad \forall V\in\R(V_0).
$$
Thus, we can conclude that
$$
\int_\Omega V_*(x)u_*^p \, \rd x = \inf\left\{L(V):V\in\overline{\R(V_0)}\right\}.
$$

{\em Step 2}. We show that $V_*$ is the unique minimizer of $L(V)$ relative to $\R(V_0)$.

Suppose that $W$ is another minimizer of $L(V)$ relative to $\R(V_0)$, then
$$
\int_\Omega V_*(x)u_*^p \,\rd x = \int_\Omega W(x)u_*^p \, \rd x.
$$
Thus
\begin{align*}
I =& \lambda(g^*,V_*)\\
=& \int_\Omega |\nabla u_*|^p + V_*(x)u_*^p \,\rd x\\ 
=& \int_\Omega |\nabla u_*|^p + W(x)u_*^p \,\rd x\\
\ge& \lambda(g^*,W)\\
\ge& I.
\end{align*}
Hence $u_*$ is the positive eigenfunction associated to $\lambda(g^*,V_*) = \lambda(g^*,W)$. Then
\begin{eqnarray}
\label{primera}
-\Delta_p u_* + V_*(x) u_*^{p-1} &=& \lambda(g^*,V_*)g^*(x)u^{p-1} \quad \mbox{in } \Omega,\\
\label{segunda}
-\Delta_p u_* + W(x) u_*^{p-1} &=& \lambda(g^*,V_*)g^*(x)u^{p-1} \quad \mbox{in } \Omega.
\end{eqnarray}
Subtracting \eqref{segunda} from \eqref{primera}, we get
$$
(V_*(x)-W(x))u_*^{p-1} =0 \quad \mbox{a.e. in } \Omega,
$$
then $V_* = W$ a.e. in $\Omega$.

Thus, by Theorem \ref{unicrearr}, there exists decreasing
function $\psi$ such that
$$
V_* = \psi(u_*) \quad \mbox{a.e. in } \Omega.
$$

{\em Step 3}. Now, we show that $g*$ is a maximizer of the linear functional
$$
H(g) := \int_\Omega g(x)u_*^p \, \rd x
$$
relative to $g\in\overline{\R(g_0)}$.

We argue by contradiction, so assume that there exists $g\in\R(g_0)$ such that
$$
\alpha = \int_\Omega g(x)u_*^p \, \rd x > \int_\Omega g^*(x) u_*^p \,\rd x = 1
$$
and take $v = \alpha^{-1/p}u_*$. Then
$$
\int_\Omega g(x)v^p \,\rd x = 1
$$
and
$$
\int_\Omega |\nabla v|^p + V_*(x)v^p \, \rd x = \frac{1}{\alpha} \int_\Omega |\nabla u_*|^p + V_*(x)u_*^p \, \rd x = \frac{1}{\alpha}\lambda(g^*,V_*) < \lambda(g^*,V_*).
$$
Therefore
$$
\lambda(g,V_*) < \lambda(g^*,V_*),
$$
which contradicts the minimality of $\lambda(g^*,V_*)$.

{\em Step 4}. Finally,  we show that $g^*$ is the unique maximizer of $H(g)$ relative to $\R(g_0)$.

Assume that there exists another maximizer $f$ of $H(g)$ relative to $\R(g_0).$ Then
$$
\int_\Omega f(x)u_*^p \, \rd x = \int_\Omega g^*(x)u_*^p \, \rd x = 1
$$
and therefore
$$
I = \lambda(g^*,V_*) \le \lambda(f,V_*) \le \int_\Omega |\nabla u|^p + V_*(x)u_*^p \, \rd x = I,
$$
then $\lambda(g^*,V_*) = \lambda(f,V_*)$ and hence $u_*$ is the eigenfunction associated to $\lambda(g^*,V_*) = \lambda(f,V_*)$. Thus 
\begin{eqnarray}
\label{primeras}
-\Delta_p u_* + V_*(x) u_*^{p-1} &=& \lambda(g^*,V_*)g^*(x)u^{p-1} \quad \mbox{in } \Omega,\\
\label{segundas}
-\Delta_p u_* + V_*(x) u_*^{p-1} &=& \lambda(g^*,V_*)f(x)u^{p-1} \quad\,\, \mbox{in } \Omega.
\end{eqnarray}
Subtracting \eqref{segundas} from \eqref{primeras}, we get
$$
\lambda(g^*,V_*)\left(g_*(x)-f(x)\right)u_*^p = 0 \quad \mbox{a.e. in } \Omega,
$$
thus $g^*=f$ a.e. in $\Omega$.

Then, by Theorem \ref{unicrearr}, there exist increasing
function $\phi$ such that
$$
g^*=\phi(u_*) \quad \mbox{ a.e. in } \Omega.
$$
This finishes the proof.
\end{proof}


\section{Differentiation of $\lambda(g,V)$}

The first aim of this section is prove the continuity of the first positive eigenvalue $\lambda(g,V)$ respect to $g$ and $V.$ Then we proceed further and compute the derivative of $\lambda(g,V)$ with respect to perturbations in $g$ and $V$.

\begin{pr}\label{continuidad}
The first positive eigenvalue $\lambda(g,V)$ of
\eqref{eigen} is continuous with respect to $(g,V)\in \mathcal A$ where
$$
\mathcal A := \{ (g,V)\in \lq\times\lq\colon (g,V) \mbox{ satisfies (H1) and (H2)}\}.
$$
i.e.,
$$
\lambda(g_n,V_n) \to \lambda(g,V),
$$
when $(g_n,V_n)\to (g,V)$ strongly in $\lq\times \lq$ and $(g_n,V_n), (g,V)\in \mathcal A$.
\end{pr}

\begin{proof} We know that
$$
\lambda(g_n,V_n) = \int_\Omega |\nabla u_n|^p + V_n(x)u_n^p \, \rd x
$$
and
$$
\lambda(g,V) = \int_\Omega |\nabla u|^p +
V(x) u^p \, \rd x,
$$
with
$$
\int_\Omega g_n(x)u_n^p \, \rd x = \int_\Omega
g(x) u^p \, \rd x = 1,
$$
where $u_n$ and $u$ are the positive eigenfunctions associated to $\lambda(g_n,V_n)$ and $\lambda(g,V)$ respectively.

We begin by observing that
$$
H(g_n) := \int_\Omega g_n(x) u^p \,\rd x = \int_\Omega (g_n(x)-g(x)) u^p \, \rd x + 1 \to 1,
$$
as $n \to \infty$. Then there exists $n_0\in \N$
such that
$$
H(g_n) > 0 \quad \forall n\ge n_0.
$$
Thus we take $v_n := H(g_n)^{-1/p} u$ and by \eqref{caract} we have
$$
\lambda(g_n,V_n) \le \int_\Omega |\nabla v_n|^p + V_n(x) v_n^p \, \rd x = \frac{1}{H(g_n)} \int_\Omega |\nabla u|^p + V_n(x) u^p \, \rd x.
$$
Therefore, taking limits when $g_n\to g$ and
$V_n\to V$ in $\lq,$ we get that
$$
\limsup_{n\to\infty}\lambda(g_n,V_n) \le \int_\Omega |\nabla u|^p + V(x) u^p \, \rd x = \lambda(g,V).
$$

On the other hand, as $V_n\to V$ strongly in $\lq$ it is easy to see that there exists $\delta_0>0$ such that
$$
\|V_n^-\|_{\lq}, \|V^-\|_{\lq} < S_{pq'} (1-\delta_0) \qquad \forall n\in\N,
$$
or
$$
V_n, V> -S_p + \delta_0 \qquad \forall n\in\N.
$$
Therefore, as $\{\lambda(g_n,V_n)\}_{n\in\N}$ is bounded, arguing as in Lemma \ref{coercividad} we have that $\{u_n\}_{n\in\N}$ is bounded in $\wpc$. Therefore there exists $v\in\wpc$ and a subsequence of $\{u_n\}_{n\in\N}$ (that we still denote by $\{u_n\}_{n\in\N}$) such that
\begin{eqnarray}
\label{weaklyd}
u_n &\rightharpoonup& v \quad \textrm{weakly in } \wpc,\\
\label{stronglyd}
u_n &\to& v \quad \textrm{strongly in } L^{pq'}(\Omega).
\end{eqnarray}
By \eqref{stronglyd} and as $g_n \to g$ in $\lq$ we have
that
$$
1 = \lim_{n\to\infty}\int_\Omega g_n(x)|u_n|^p \, \rd x
= \int_\Omega g(x)|v|^p \, \rd x.
$$
Finally, by \eqref{weaklyd}, \eqref{stronglyd} and, as $V_n \to V$ in $\lq$ we arrive at
\begin{align*}
\liminf_{n\to\infty}\lambda(g_n,V_n) &= \liminf_{n\to\infty} \int_\Omega|\nabla u_n|^p + V_n(x)u_n^p \, \rd x\\
&\ge \int_\Omega |\nabla v|^p+ V(x)|v|^p \, \rd x\\
&\ge \lambda(g,V)
\end{align*}
and the result follows.
\end{proof}

\begin{ob}
Observe that if instead of (H2) we required {\em only} that $V>-S_p + \delta$, the exact same proof of Proposition \ref{continuidad} gives the continuity of $\lambda(g,V)$ with respect to weak convergence.
\end{ob}

Now we arrive at the main result of the section, namely we compute the derivative of the first positive eigenvalue
$\lambda(g,V)$ with respect to perturbations in $g$ and $V$.

We begin by describing the kind of variations that we are going to consider. Let $W$ be a regular (smooth) vector field, globally Lipschitz, with support in $\Omega$ and let $\varphi_t:\rn\to\rn$ be the flow defined by
\begin{equation}\label{cv}
\begin{cases}
\frac{\rd}{\rd t}\varphi_t(x) = W(\varphi_t(x)) & t>0,\\
\varphi_0(x) = x & x\in\rn.
\end{cases}
\end{equation}
We have
$$
\varphi_t(x) = x+tW(x)+o(t) \quad\forall x\in\rn.
$$
Thus, if $g$ and $V$ are measurable functions that satisfy the assumptions (H1) and (H2), we define $g_t := g\circ\varphi_t^{-1}$ and $V_t := V\circ\varphi_t^{-1}$. Now, let
$$
\lambda(t) := \lambda(g_t,V_t)=\int_\Omega|\nabla u_t|^p + V_t(x) |u_t|^p \, \rd x,
$$
with
$$
\int_\Omega g_t(x)u_t^p \, \rd x = 1,
$$
where $u_t$ is the eigenfunction associated to $\lambda(t)$.

\begin{ob}\label{divW}
In order to this approach to be usefull for the optimization problem of the previous section, we need to guaranty that $g_t\in\R(g_0)$ and $V_t\in \R(V_0)$ whenever $g\in\R(g_0)$ and $V\in \R(V_0)$.

It is not difficult to check that this is true for incompressible deformation fields, i.e., for those $W$'s such that
$$
\div W = 0.
$$
\end{ob}

\begin{lem}\label{convergencia}
Given $f\in\lq$ then
$$
f_t := f\circ\varphi_t^{-1}\to f \textrm{ in } \lq, \mbox{ as } t\to 0^+.
$$
\end{lem}

\begin{proof}
Let $\varepsilon>0$ and let $g\in C_c^\infty(\Omega)$ fixed such that $\|f-g\|_{\lq}<\varepsilon$. By the usual change of variables formula, we have,
$$
\|f_t - g_t\|^q_{\lq} = \int_{\Omega} |f-g|^q J\varphi_t \, \rd x,
$$
where $g_t = g\circ\varphi_t^{-1}$ and $J\varphi_t$ is the Jacobian of $\varphi_t$. We know that
$$
J\varphi_t = 1 + t \div W + o(t).
$$
Here $\div W$ is the divergence of $W$. Then
$$
\|f_t - g_t\|^q_{\lq} = \int_{\Omega}|f-g|^q (1 + t \div W + o(t))\, \rd x.
$$
Then, there exist $t_1>0$ such that if $0<t<t_1$ then
$$
\|f_t-g_t\|_{\lq} < C\varepsilon,
$$
where $C$ is a constant independent of $t$. Moreover, since $\varphi_t^{-1} \to Id$ in the $C^1$ topology when $t\to0$ then $g_t = g\circ\varphi_t^{-1}\to g$ in the $C^1$ topology and therefore there exist $t_2>0$ such that if $0<t<t_2$ then
$$
\|g_t -g\|_{\lq} < \varepsilon.
$$
Finally, we have for all $0<t<t_0=\min\{t_1,t_2\}$ then
$$
\|f_t-f\|_{\lq} \le \|f_t-g_t\|_{\lq} + \|g_t-g\|_{\lq} + \|f-g\|_{\lq} \le C\varepsilon,
$$
where $C$ is a constant independent to $t$.
\end{proof}

By Proposition \ref{continuidad} and Lemma \ref{convergencia} we have that

\begin{co}\label{conv}
Let $g$ and $V$ be measurable functions that satisfy the
assumptions (H1) and (H2). Then, with the previous
notation, $\lambda(t)$ is continuous at $t=0,$ i.e.,
$$
\lambda(t) \to \lambda(0) = \lambda(g,V)\quad \mbox{as }
t\to0^+.
$$
\end{co}

\begin{lem}\label{uconv}
Let $g$ and $V$ be measurable functions that satisfy the
assumptions (H1) and (H2). Let $u_t$ be the normalized positive eigenfunction associated to $\lambda(t)$ with $t>0$. Then 
$$
\lim_{t\to 0^+} u_t = u_0 \quad \mbox{strongly in } \wpc.
$$
where $u_0$ is the unique normalized positive eigenfunction associated to $\lambda(g,V)$.
\end{lem}

\begin{proof}
Form the previous corollary we deduce that $\lambda(t)$ is bounded and, as in the proof of Proposition \ref{continuidad}, we further deduce that $\{u_t\}$ is bounded in $\wpc$.

So, given $\{t_n\}_{n\in\N}$, we have that $\{u_{t_n}\}_{n\in\N}$ is bounded in $\wpc$ and therefore there exists $u_0\in\wpc$ and some subsequence (still denoted by $\{u_{t_n}\}_{n\in\N}$) such that
\begin{eqnarray}
\label{weaklyt}
u_{t_n} &\rightharpoonup& u_0 \quad \mbox{weakly in } \wpc,\\
\label{stronglyt}
u_{t_n} &\to& u_0 \quad \mbox{strongly in } L^{pq'}(\Omega).
\end{eqnarray}
Since $(g_{t_n}, V_{t_n})\to (g,V)$ strongly in $\lq\times\lq$ as $n\to\infty$ and by \eqref{stronglyt} we get
$$
1 = \lim_{n\to\infty}\int_\Omega g_{t_n}(x)|u_{t_n}|^p \, \rd x = \int_\Omega g(x)|u_0|^p \, \rd x
$$
and
$$
\lim_{n\to\infty}\int_\Omega V_{t_n}(x)|u_{t_n}|^p \, \rd
x = \int_\Omega	V(x)|u_0|^p \, \rd x.
$$
Thus, using \eqref{weaklyt},
\begin{align*}
\lambda(0) =& \lim_{n\to\infty}\lambda(t_n)\\
=& \lim_{n\to\infty}\int_\Omega|\nabla u_{t_n}|^p +
V_{t_n}(x)|u_{t_n}|^p \, \rd x\\
\ge& \int_\Omega |\nabla u_0|^p + V(x)|u_0|^p \, \rd x\\
\ge& \lambda(0),
\end{align*}
then $u_0$ is the a normalized eigenfunction associated to $\lambda(0)$ and, as $\{u_{t_n}\}_{n\in\N}$ are positive, it follows that $u_0$ is positive.

Moreover
$$
\|\nabla u_{t_n}\|_{L^p(\Omega)}\to\|\nabla
u_0\|_{L^p(\Omega)} \quad \mbox{as } n\to\infty.
$$
Then, using again \eqref{weaklyt}, we have
$$
u_{t_n}\to u_0 \quad \textrm{in } \wpc \mbox{ as } n \to \infty.
$$
as we wanted to show.
\end{proof}

\begin{ob}\label{remark}
It is easy to see that, as $\varphi_t\to Id$ in the $C^1$
topology, then from Lemma \ref{uconv} it follows that
$$
u_t\circ\varphi_t \to u_0 \quad \mbox{strongly in } \wpc \mbox{ as } t\to 0,
$$
when $u_t\to u_0$ strongly in $\wpc$.
\end{ob}

Now, we arrive at the main result of the section
\begin{te} \label{lambda'}
With the previous notation, if $g$ and $V$ are measurable functions that satisfy the assumptions (H1) and (H2), we have that $\lambda(t)$ is differentiable at $t=0$ and
\begin{align*}
\dfrac{\rd\lambda(t)}{\rd t}\Big|_{t=0} =& \int_\Omega\left(|\nabla u_0|^p + V(x)|u_0|^p\right) \div W\,\rd x - p\int_\Omega|\nabla u_0|^{p-2} \langle\nabla u_0, ^TW'\nabla u_0 ^T\rangle \,\rd x\\
& -\lambda(0) \int_\Omega g(x)|u_0|^p \div W \, \rd x,
\end{align*}
where $W'$ denotes the differential matrix of $W,$ $^TA$ is the transpose of the matrix $A$ and $u_0$ is the eigenfunction associated to $\lambda(0)=\lambda(g,V)$.
\end{te}

\begin{proof}
First we consider $v_t := u_0\circ\varphi_t^{-1}$. Then, by the change of variables formula we get
\begin{align*}
\int_\Omega g_t(x)|v_t|^p \, \rd x =& \int_\Omega g(x)|u_0|^p J\varphi_t 
\, \rd x\\
=& \int_\Omega g(x)|u_0|^p (1 + t \div W + o(t))\, \rd x\\
=& 1 + t\int_\Omega g(x)|u_0|^p \div W\,\rd x + o(t),\\
\int_\Omega V_t(x)|v_t|^p \, \rd x =& \int_\Omega V(x)|u_0|^p J\varphi_t \, \rd x\\
=& \int_\Omega V(x)|u_0|^p (1 + t \div W+o(t))\, \rd x\\
=& \int_\Omega V(x)|u_0|^p \, \rd x + t\int_\Omega V(x) |u_0|^p \div W\, \rd x + o(t)
\end{align*}
and
\begin{align*}
\int_\Omega |\nabla v_t|^p\, \rd x =&  \int_\Omega |^{T}[\varphi_t']^{-1}(x) \nabla u_0^T|^p J\varphi_t\, \rd x\\
=& \int_\Omega |(I-t ^TW'+ o(t)) \nabla u_0^T|^p (1 + t \div W+o(t))\, \rd x\\
=& \int_\Omega (|\nabla u_0|^p - tp|\nabla u_0|^{p-2} \langle\nabla u_0, ^T W'\nabla u_0^T\rangle) (1 + t \div W)\, \rd x + o(t)\\
=& \int_\Omega |\nabla u_0|^p \, \rd x + t \int_\Omega |\nabla u_0|^p \div W\, \rd x\\
&- tp \int_\Omega |\nabla u_0|^{p-2} \langle\nabla u_0, ^T W'\nabla u_0^T\rangle\, \rd x +o(t).
\end{align*}
Then, for $t$ small enough,
$$
\int_\Omega g_t(x) |v_t|^p \, \rd x > 0
$$
and therefore
$$
\lambda(t)\le\dfrac{\int_\Omega|\nabla v_t|^p V_t(x)+|v_t|^p \,
\rd x}{\int_\Omega g_t(x) |v_t|^p\, \rd x}.
$$
So
$$
\lambda(t) \int_\Omega g_t(x) |v_t|^p\, \rd x \le
\int_\Omega|\nabla v_t|^p V_t(x)|v_t|^p \, \rd x,
$$
then we have that
\begin{align*}
\lambda(t)\left(1 + t\int_\Omega g(x) |u_0|^p \div W\,\rd x\right) \le& \int_\Omega |\nabla u_0|^p +
V(x)|u_0|^p \, \rd x\\
&+ t \int_\Omega\left( |\nabla u_0|^p + V(x) |u_0|^p\right) \div W\, \rd x\\
&- tp \int_\Omega |\nabla u_0|^{p-2} \langle\nabla u_0, ^T W'\nabla u_0^T\rangle\, \rd x +o(t)\\
=& \lambda(0)+ t \int_\Omega\left( |\nabla u_0|^p + V(x) |u_0|^p\right) \div W\, \rd x\\
&- tp \int_\Omega |\nabla u_0|^{p-2} \langle\nabla u_0, ^T W'\nabla u_0^T\rangle\, \rd x +o(t)
\end{align*}
and we get that
\begin{align*}
\frac{\lambda(t)-\lambda(0)}{t}\le &
\int_\Omega (|\nabla u_0|^p + V(x) |u_0|^p) \div W\, \rd x\\
&- p \int_\Omega |\nabla u_0|^{p-2} \langle\nabla
u_0, ^T W'\nabla u_0^T\rangle\, \rd x\\
& - \lambda(t)\int_\Omega g(x)|u_0|^p \div W\,\rd x
+ O(t).
\end{align*}
In a similar way, if we take $w_t=u_t\circ\varphi_t$ we have that
\begin{align*}
\frac{\lambda(t)-\lambda(0)}{t} \ge& \int_\Omega
(|\nabla w_t|^p + V(x)|w_t|^p) \div W \, \rd x\\
& - p\int_\Omega|\nabla w_t|^{p-2}\langle\nabla w_t, ^TW'\nabla w_t^T\rangle \, \rd x\\
& - \lambda(0)\int_\Omega g(x)|w_t|^p \div W \, \rd x + O(t).
\end{align*}
Thus, taking limit in the two last inequalities as $t\to 0^+,$ by the Corollary \ref{conv} and Remark \ref{remark}, we get that
\begin{align*}
\lim_{t\to 0^+} \frac{\lambda(t)-\lambda(0)}{t}=& \int_\Omega\left(|\nabla u_0|^p + V(x)|u_0|^p\right) \div W\, \rd x\\
& - p\int_\Omega|\nabla u_0|^{p-2}\langle\nabla u_0, ^TW'\nabla u_0 ^T\rangle \,\rd x\\
& - \lambda(0)\int_\Omega g(x)|u_0|^p \div W \, \rd x.
\end{align*}
This finishes the proof.
\end{proof}

\begin{ob}\label{lambda'divW}
When we work in the class of rearrangements of a fixed pair $(g_0, V_0)$, as was mentioned in Remark \ref{divW}, we need the deformation field $W$ to verified $\div W = 0$. So, in this case, the formula for  $\lambda'(0)$ reads,
$$
\dfrac{\rd\lambda(t)}{\rd t}\Big|_{t=0} =
-p\int_\Omega|\nabla u_0|^{p-2}\langle\nabla u_0,
^TW'\nabla u_0 ^T\rangle \,\rd x.
$$
\end{ob}

In order to improve the expression for the formula of
$\lambda'(0)$ we need a lemma that will allow us to regularized
problem \eqref{eigen} since solutions to \eqref{eigen} are $C^{1,\delta}$ for some $\delta>0$ but are not $C^2$ nor $W^{2,q}$ in general (see \cite{T}).

\begin{lem}\label{regularizacion}
Let $V, g$ be measurable functions that satisfy the assumptions (H1) and (H2) and let $V_\ep, g_\ep\in C_0^\infty(\Omega)$ be such that $V_\ep\to V$ and $g_\ep\to g$ in $L^q(\Omega)$. Let
$$
\lambda_\ep := \min_{\genfrac{}{}{0cm}{}{u\in\wpc}{\int_\Omega g_\ep(x)|v|^p \, \rd x = 1}}\int_\Omega (|\nabla v|^2 + \ep^2)^{(p-2)/2} |\nabla v|^2 + V_\ep(x) |v|^p\, \rd x.
$$
Finally, let $u_\ep$ be the unique normalized positive eigenfunction associated to $\lambda_\ep$.

Then, $\lambda_\ep\to \lambda(g,V)$ and $u_\ep \to u$ strongly in $\wpc$ where $u_0$ is the unique normalized positive eigenfunction associated to $\lambda(g,V)$.
\end{lem}

\begin{proof}
First, observe that, as $g_\ep\to g$ in $L^q(\Omega)$ if $u_0$ is the normalized positive eigenfunction associated to $\lambda(g,V)$ we have that
$$
\int_\Omega g_\ep(x) |u_0|^p\, \rd x >0.
$$
Then, taking $u_0$ in the characterization of $\lambda_\ep$ we get
$$
\lambda_\ep \le \dfrac{\int_\Omega (|\nabla u_0|^2 + \ep^2)^{(p-2)/2} |\nabla u_0|^2 + V_\ep(x) |u_0|^p\, \rd x} {\int_\Omega g_\ep(x)|u_0|^p\, \rd x}.
$$
Hence, passing to the limit as $\ep\to 0^+$ we arrive at
$$
\limsup_{\ep\to 0^+} \lambda_\ep \le \lambda(g,V).
$$

Now, for any $v\in\wpc$ normalized such that
$$
\int_{\Omega} g_\ep(x) |v|^p\, \rd x = 1,
$$
we have that
$$
\int_\Omega (|\nabla v|^2 + \ep^2)^{(p-2)/2} |\nabla v|^2 + V_\ep(x) |v|^p\, \rd x\ge \int_\Omega |\nabla v|^p + V_\ep(x) |v|^p\, \rd x\ge \lambda(g_\ep, V_\ep),
$$
therefore $\lambda_\ep\ge \lambda(g_\ep, V_\ep)$.

Now, by Proposition \ref{continuidad}, we have that $\lambda(g_\ep, V_\ep)\to \lambda(g,V)$ as $\ep\to 0^+$. So
$$
\liminf_{\ep\to 0^+}\lambda_\ep\ge \lambda(g,V).
$$

Finally, from the convergence of the eigenvalues, it is easy to see that the normalized eigenfunctions $u_\ep$ associated to $\lambda_\ep$ are bounded in $\wpc$ uniformly in $\ep>0$. Therefore, there exists a sequence,
that we still call $\{u_\ep\}$, and a function $u\in\wpc$ such that
\begin{align*}
& u_\ep \rightharpoonup u \qquad \mbox{weakly in }\wpc,\\
& u_\ep \to u \qquad \mbox{strongly in } L^{pq'}(\Omega).
\end{align*}
Recall that our assumptions on $q$ imply that $pq'<p^*$.

Hence,
$$
\int_\Omega g(x) |u|^p\, \rd x = \lim_{\ep\to 0^+} \int_\Omega g_\ep(x) |u_\ep|^p\, \rd x = 1,
$$
and so
\begin{align*}
\lambda(g,V) =& \lim_{\ep\to 0^+} \lambda_\ep\\
=& \lim_{\ep\to 0^+}\int_\Omega (|\nabla u_\ep|^2 + \ep^2)^{(p-2)/2} |\nabla u_\ep|^2 + V_\ep(x) |u_\ep|^p\, \rd x\\
\ge& \int_\Omega |\nabla u|^p + V(x) |u|^p\, \rd x\\
\ge& \lambda(g,V).
\end{align*}

These imply that $u = u_0$ the unique normalized positive eigenfunction associated to $\lambda(g,V)$ and that $\|u_\ep\|_{\wpc}\to\|u\|_{\wpc}$ as $\ep\to 0^+$. So
$$
u_\ep \to u_0\qquad \mbox{strongly in }\wpc.
$$
This finishes the proof.
\end{proof}

\begin{ob}\label{ecu.regularizadas}
Observe that the eigenfunctions $u_\ep$ are weak solutions to
$$
\begin{cases}
-\div((|\nabla u_\ep|^2 + \ep^2)^{(p-2)/2} \nabla u_\ep) + V_\ep(x) |u_\ep|^{p-2}u_\ep = \lambda_\ep g_\ep(x) |u_\ep|^{p-2}u_\ep & \mbox{in }\Omega,\\
u = 0 & \mbox{on }\partial\Omega.
\end{cases}
$$
Therefore, by the classical regularity theory (see \cite{LU}), the functions $u_\ep$ are $C^{2,\delta}$ for some $\delta>0$.
\end{ob}

With these preparatives we can now prove the following Theorem.
\begin{te}\label{teo.principal}
With the assumptions and notations of Theorem \ref{lambda'}, we have that
$$
\frac{\rd \lambda(t)}{\rd t}\Big|_{t=0} = \lambda'(0) = \int_{\Omega} (V(x) - \lambda(0) g(x)) \div(|u_0|^pW)\, \rd x,
$$
for every field $W$ such that $\div W=0$.
\end{te}

\begin{proof}
During the proof of the Theorem, we will required the eigenfunction $u_0$ to be $C^2$. As it is well known (see \cite{T}), this is not true.

In order to overcome this difficulty, we regularize the problem and work with the regularized eigenfunctions $u_\ep$ defined in Lemma \ref{regularizacion}.

Since in the resulting formula only appears up to the first derivatives of $u_\ep$ and $u_\ep\to u_0$ strongly in $\wpc$ the result will follows.

For the sake of simplicity, we choose to work formally with $u_0$. The changes in order to make this argument rigurouse are straight forward.

Let $W\in C^1_0(\Omega; \rn)$. Then, we have that
$$
\int_\Omega \div(|\nabla u_0|^p W)\, \rd x = 0.
$$
So,
\begin{align*}
p\int_\Omega|\nabla u_0|^{p-2}\langle\nabla u_0,
^TW'\nabla u_0 ^T\rangle \,\rd x =& p\int_\Omega|\nabla u_0|^{p-2}\langle\nabla u_0, ^TW'\nabla u_0 ^T\rangle \,\rd x\\
&+ \int_\Omega \div(|\nabla u_0|^p W)\, \rd x\\
=& p\int_\Omega|\nabla u_0|^{p-2}\langle\nabla u_0,
^TW'\nabla u_0 ^T\rangle \,\rd x\\
&+ p\int_\Omega |\nabla u_0|^{p-2} \langle \nabla u_0, D^2 u_0 W^T\rangle\, \rd x\\
=& p\int_\Omega |\nabla u_0|^{p-2} \langle \nabla u_0, ^TW'\nabla u_0^T + D^2 u_0 W^T\rangle\, \rd x\\
=& p\int_\Omega |\nabla u_0|^{p-2} \langle \nabla u_0, \nabla\langle \nabla u_0, W\rangle\rangle\, \rd x.
\end{align*}
Now, we use the fact that $u_0$ is a weak solution to
\eqref{eigen} to get
\begin{align*}
\lambda'(0) =& p\int_\Omega (V(x) - \lambda(0)g(x)) |u_0|^{p-2} u_0 \langle \nabla u_0, W\rangle\, \rd x\\
=& \int_{\Omega} (V(x) - \lambda(0)g(x)) \div(|u_0|^p W)\, \rd x.
\end{align*}
The proof is now complete.
\end{proof}

\bibliographystyle{plain}
\def\cprime{$'$}

\end{document}